\documentclass[11pt]{article}
\usepackage[a4paper, margin=1in]{geometry}
\usepackage{amsmath, amssymb, amsthm, mathtools, bm}
\usepackage{algorithm}
\usepackage{algpseudocode}
\usepackage{enumitem}
\usepackage[T1]{fontenc}
\usepackage{lmodern}
\usepackage{microtype}
\usepackage[hidelinks]{hyperref}

\newtheorem{theorem}{Theorem}[section]
\newtheorem{lemma}[theorem]{Lemma}
\newtheorem{proposition}[theorem]{Proposition}
\newtheorem{assumption}{Assumption}
\newtheorem{corollary}[theorem]{Corollary}
\newtheorem{definition}[theorem]{Definition}
\theoremstyle{remark}
\newtheorem{remark}[theorem]{Remark}

\newcommand{\R}{\mathbb{R}}
\newcommand{\Lzero}{\mathcal{L}_0}
\newcommand{\Ugap}{\mathcal{U}}
\newcommand{\Deltam}{\Delta_m}
\newcommand{\argmin}{\mathop{\rm argmin}}
\newcommand{\argmax}{\mathop{\rm argmax}}
\newcommand{\OO}{\mathcal{O}}

\title{Global \(o(1/k^2)\) Merit Complexity of Regularized Newton Methods for Convex Multiobjective Optimization}
\author{
 Yuqia Wu\footnote{School of Mathematical Sciences, Shenzhen University, 
 	China (\href{mailto:mayuqiawu@szu.edu.cn}{mayuqiawu@szu.edu.cn}).} ,\ \ 
 Yue Wang\footnote{Independent Researcher, ORCID:0000-0001-7988-8791(\href{mailto:yuewangm@outlook.com}{yuewangm@outlook.com}).},\ \
 Yaohua Hu\footnote{School of Mathematical Sciences, Shenzhen University, 
 	China (\href{mailto:mayhhu@szu.edu.cn}{mayhhu@szu.edu.cn}).}
 }
\date{}

\begin{document}
    \maketitle

    \begin{abstract}
    We investigate a regularized Newton method for unconstrained convex multi-objective optimization with twice continuously differentiable objectives whose Hessians are Lipschitz continuous. At each iteration, the method minimizes the quadratically regularized max-envelope of the local quadratic models.  Using a Tanabe-type merit function, we prove that this merit decays at the global asymptotic rate $o(1/k^2)$ under the compactness assumption on the initial component-wise lower level set. This result also covers the single-objective case as a special case. Finally, we construct an explicit one-dimensional convex bi-objective family showing that no uniform merit estimate of order $\mathcal O(k^{-(2+\delta)})$ can hold for any fixed $\delta>0$. Thus the exponent $2$ is essentially sharp in the uniform polynomial sense, despite the $o(1/k^2)$ decay on each fixed trajectory.
    \end{abstract}

    \section{Introduction}

    We consider the unconstrained convex multi-objective optimization problem
    \begin{equation}\label{model}
        \min_{x\in\R^n}F(x):=(f_{1}(x),\dots,f_{m}(x)),
    \end{equation}
    where each component function $f_i:\R^n\to\R$ is convex and twice continuously differentiable, and $\nabla^2 f_i$ is Lipschitz continuous with modulus $H_i>0$. We write 
    $$H:= \max_{i\in [m]} H_i.$$
    Problems of the form \eqref{model} arise when several competing criteria must be optimized simultaneously, and the natural goal is to compute a weak Pareto optimal point rather than a minimizer of a prescribed scalarization.

    For smooth multi-objective optimization, global complexity theory has been developed for first-order descent methods. A foundational contribution is the steepest descent method by Fliege and Svaiter \cite{fliege2000}, which computes a common descent direction by solving a convex subproblem and does not require a fixed scalarization in advance. More precisely, at a current iterate $x^k$, the steepest descent method updates
    \[
    x^{k+1} = x^k + \alpha_k d^k \ \ {\rm with} \ \ d^k\in \argmin_{d\in\R^n} \left\{ \max_{i\in[m]}\langle \nabla f_i(x^k),d\rangle        +\frac12\|d\|^2 \right\},
    \]
    where $\alpha_k>0$ is chosen by a suitable stepsize rule. The subproblem selects a direction that balances descent among all component functions and preserves the scalarization-free nature of the method. Later work clarified the worst-case complexity of this method. In particular, Fliege, Vaz, and Vicente \cite{fliege2019} showed that multi-objective gradient descent has complexity bounds parallel to those of scalar gradient descent, including an $\OO(k^{-1/2})$ stationarity rate for nonconvex objectives, an $\OO(k^{-1})$ rate in the convex case, and a linear rate under strong convexity.

    This first-order min-max viewpoint has also been extended beyond the smooth setting. For composite multi-objective problems, Tanabe, Fukuda, and Yamashita \cite{tanabe19} introduced proximal gradient methods, and provided rates of order $\OO(k^{-1/2})$, $\OO(k^{-1})$, and linear convergence based on suitable merit functions in nonconvex, convex, and strongly convex or PL-type settings, respectively \cite{tanabe23,tanabe24}. A distinctive feature of accelerated multi-objective methods is that the active convex combination of objective gradients may vary along the iterates. El Moudden and El Mouatasim \cite{elmoudden2021} obtained an $\OO(k^{-2})$ rate for an accelerated diagonal steepest descent method under an eventual fixed-multiplier assumption. Tanabe, Fukuda, and Yamashita \cite{tanabe2023} later proposed an accelerated proximal gradient method and proved a global $\OO(k^{-2})$ rate for a merit function under less restrictive assumptions. This rate was refined by Zhang and Yang \cite{zhang2023convergence}, which showed that the convergence rate of the multi-objective accelerated proximal gradient method can be improved from $\OO(k^{-2})$ to $o(k^{-2})$. Related inertial and Nesterov-type approaches obtain $\OO(k^{-2})$- and $\OO(t^{-2})$-type rates in discrete and continuous-time settings, respectively \cite{sonntag2024jota,sonntag2024siopt}. These results indicated that the merit complexity is now well developed for first-order multi-objective methods, especially in the convex setting.

    By contrast, although second-order methods for multi-objective optimization have been studied for a long time, their global merit complexity theory is far from completion. Fliege, Gra\~na Drummond, and Svaiter \cite{fliege09} proposed a Newton method whose direction is obtained from minimizing a max-envelope of the quadratic models of the component functions, and proved the convergence to Pareto critical points, as well as local quadratic convergence under additional strongly convex assumptions. Quasi-Newton methods were developed to reduce the cost of Hessian evaluations, including the methods of Qu, Goh, and Chan \cite{qu2011}, Povalej \cite{povalej2014}, and recent BFGS-type schemes \cite{prudente2024} for nonconvex multi-objective problems. Regularization has also been explored at this line of work. For convex multi-objective optimization, Wang and Liu \cite{wang2012regularized} proposed a regularized Newton method and proved the global convergence under compact level-set assumptions. Recently, cubic-regularization methods have provided stationarity-type guarantees for nonconvex multiobjective optimization, including $O(k^{-2/3})$-type rates in \cite{ghosh2025,goncales2025cubic}, and the latter also covers inexact derivative information and finite-difference implementations.

    These results leave open a different question in the convex setting. In scalar convex optimization, a central complexity measure is the objective gap $f(x^k)-f^\star$, not only stationarity residuals. For multi-objective problems, an analogous role is played by merit functions \cite{tanabe24}, which vanish precisely on the weak Pareto solution set. The available second-order theory controls convergence to Pareto criticality, local behavior, or the number of iterations needed to make a stationarity residual small, but it does not provide decay estimates for such merit functions along the generated iterates. This is the complexity gap addressed in the present paper.

    The scalar optimization theory provides a useful reference. For smooth convex optimization, gradient descent has the classical $\OO(k^{-1})$ rate, while Nesterov's accelerated method has $\OO(k^{-2})$ rate \cite{nesterov1983}. In the composite case, FISTA gives the same order \cite{beck2009}. For second-order methods with Lipschitz continuous Hessians, the cubic regularized Newton method of Nesterov and Polyak \cite{nesterov2006cubic} achieves a global $\OO(k^{-2})$ rate on convex problems. Recently, Mishchenko \cite{mishchenko2023regularized} showed that a regularized Newton method with a square-root regularization rule also attains a global $\OO(k^{-2})$ rate.
    Compared with the mature scalar theory, second-order methods for convex multi-objective optimization still lack sharp global rates for merit functions, which play the role of objective gaps in the scalar case.  

    Motivated by this observation, the present paper develops a multi-objective regularized Newton method. At iteration $k$, the method computes $d^k$ by minimizing the max-envelope of the local quadratic models, regularized by $\eta_k\|d\|^2/2$, where $\eta_k$ is related to the current multi-objective stationarity measure.
   A central issue in the analysis of multi-objective methods is the choice of a scalar merit function, because the objective value is vector-valued and therefore does not by itself define a scalar optimality gap. We use the Tanabe-type merit function \cite{tanabe24}:
    \begin{equation}\label{eq:def_U}
    \Ugap(x):=
    \sup_{z\in\R^n}
    \min_{i\in[m]}\big\{f_i(x)-f_i(z)\big\}.
    \end{equation}
    This function is exact for weak Pareto optimality in the sense that $\Ugap(x)\ge0$ and $\Ugap(x)=0$ if and only if $x$ is a weak Pareto optimal point; see \cite[Theorem~3.1]{tanabe24}. Therefore, complexity estimates for $\Ugap(x^k)$ play the role of objective gap estimates in scalar convex optimization. 

    The paper makes two main contributions.

    \begin{itemize}
        \item[{\rm (i)}] We prove the global merit complexity estimate
        \[
            \mathcal U(x^k)=o(1/k^2)
        \]
        for the regularized Newton method. This gives a second-order counterpart to the accelerated first-order merit rates known for convex multi-objective optimization \cite{tanabe2023,zhang2023convergence}. In particular, when $m=1$, the method reduces to the square-root regularized Newton method studied by Mishchenko~\cite{mishchenko2023regularized}, and the merit function $\Ugap$ reduces to the usual scalar objective gap. Thus, our result recovers the scalar $\OO(1/k^2)$ bound and further refines it to $o(1/k^2)$.

        Moreover, it is worth noting that the analysis is not a direct extension of the scalar argument of Mishchenko~\cite{mishchenko2023regularized}. In the multi-objective setting, there is no distinguished objective gap $f(x^k)-f^\star$, and the active convex combination of component models may vary along the trajectory. These features require a separate merit analysis for the multi-objective method.
    
    \item[{\rm (ii)}] We show that the polynomial exponent $2$ is essentially sharp in a uniform worst-case sense. We construct an explicit one-dimensional convex bi-objective family, and show that no estimate of the form
    \[
    \Ugap(x^k)=\OO(k^{-(2+\delta)})
    \]
    can hold uniformly over the family for any $\delta>0$. Hence the exponent $2$ cannot be improved uniformly, even though the actual rate is $o(1/k^2)$.
    \end{itemize}

    The remainder of the paper is organized as follows. Section~\ref{sec:preliminaries} introduces the notations and collects preliminary facts on weak Pareto optimality and the stationarity measure. Section~\ref{sec:algorithm} presents the multi-objective regularized Newton method, establishes its descent estimates, and proves the global $o(1/k^2)$ merit rate. Section~\ref{sec:barrier} constructs an explicit barrier family and shows the uniform sharpness of the polynomial exponent $2$.
  
    \section{Preliminaries}
    \label{sec:preliminaries}
    \subsection{Notations}
    Throughout this paper, $\R^n$ denotes the $n$-dimensional Euclidean space equipped with the standard inner product $\langle\cdot,\cdot\rangle$ and the induced norm $\|\cdot\|$. We write $\mathbb N:=\{0,1,2,\ldots\}$ and $\R_+:=[0,\infty)$.
    For a positive integer $m$, let $[m]:=\{1,2,\ldots,m\}$. For vectors $u,v\in\R^m$, the inequalities $u\le v$ and $u<v$ are understood component-wise, that is, $u_i\le v_i$ and $u_i<v_i$ for all $i\in[m]$, respectively. For matrices, $\|\cdot\|$ denotes the induced operator norm, and $A\succeq0$ means that the symmetric matrix $A$ is positive semidefinite.
    Let
    \[
    \Deltam:=\left\{\lambda\in\R_+^m:\ \sum_{i=1}^m\lambda_i=1\right\}
    \]
    be the unit simplex in $\R^m$. 
    For a set $S\subset\R^n$, $\operatorname{conv}(S)$ denotes its convex hull. 
    Given an initial point $x^0\in\R^n$, the component-wise lower level set of $F$ at $x^0$ is
    \[
    \operatorname{Lev}_F(x^0) := \{x\in\R^n:\ f_i(x)\le f_i(x^0),\ \forall i\in[m]\}.
    \]

    \subsection{Basic properties and Pareto optimality}
    We collect two elementary consequences of the Hessian Lipschitz continuity of the component functions. These estimates will be used in the analysis of the regularized Newton step. Recall that $H>0$ is chosen as a common Hessian Lipschitz constant, namely
    \[
    \|\nabla^2 f_i(x)-\nabla^2 f_i(y)\| \le H\|x-y\|, \quad \forall x,y\in\R^n,\ i\in[m].
    \]
    We recall the following two inequalities established in \cite[Lemma 1]{nesterov2006cubic}.
    \begin{lemma}
        \label{lem:upperbound} For every $x, d\in \mathbb{R}^{n}$ and $i\in [m]$,
        it holds that
        \begin{align}
            f_{i}(x+d) \le f_{i}(x)+\langle\nabla f_{i}(x),d\rangle+\frac{1}{2} d^{\top}\nabla^{2} f_{i}(x)d+\frac{H}{6}\|d\|^{3}, \label{eq:func_upper} \\
            \|\nabla f_{i}(x+d)-\nabla f_{i}(x)-\nabla^{2} f_{i}(x)d\|\le \frac{H}{2}\|d\|^{2}. \label{eq:grad_remainder}
        \end{align}
    \end{lemma}

    We next recall the basic notions of optimality for the multi-objective problem \eqref{model}.
    \begin{definition}
    A point $x^\star\in\R^n$ is called a \emph{Pareto optimal point} if there is no $y\in\R^n$ such that 
    \[
    F(y)\le F(x^\star) \quad {\rm and} \quad  F(y)\ne F(x^\star).
    \]
    A point $x^\star\in\R^n$ is called a \emph{weak Pareto optimal point} if there is no $y\in\R^n$ such that
    \[
    F(y)<F(x^\star).
    \]
    We denote the weak Pareto solution set by $\mathcal{X}_w$.
    \end{definition}

   The following characterization states that weak Pareto optimality is equivalent to the nonexistence of a direction that strictly decreases all first-order linearizations; see, e.g., \cite{fliege2000,fliege09}.

\begin{lemma}\label{lem:weak_pareto_characterization}
For any $x\in\R^n$, the following statements are equivalent:
\begin{enumerate}[label={\rm (\roman*)}]
    \item $x$ is a weak Pareto optimal point of \eqref{model}.
    \item there is no direction $d\in\R^n$ such that $\langle \nabla f_i(x),d\rangle <0, \ \forall i\in[m].$
    \item There exists $\lambda\in\Deltam$ such that $\sum_{i=1}^m \lambda_i\nabla f_i(x)=0$. Equivalently, $$
        0\in \operatorname{conv}\{\nabla f_i(x): i\in[m]\}.$$
\end{enumerate}
\end{lemma}

We also use the stationarity measure and the associated set of minimizing
multipliers
\begin{equation}
    \label{eq:def_s}
    s(x):= \min_{\lambda\in\Delta_m} \left\|\sum_{i=1}^{m}\lambda_i\nabla f_i(x)\right\|, \qquad \Lambda(x):= \argmin_{\lambda\in\Delta_m} \left\|\sum_{i=1}^{m}\lambda_i\nabla f_i(x)\right\|.
\end{equation}
The value $s(x)$ is the distance from the origin to $\operatorname{conv}\{\nabla f_i(x):i\in[m]\}$. Hence, by Lemma~\ref{lem:weak_pareto_characterization}, $s(x)=0$ if and only if $x$ is weak Pareto optimal. The set $\Lambda(x)$ is nonempty because $\Deltam$ is compact.

The continuity of $s(\cdot)$ is a standard consequence of the continuity theorem for marginal functions over compact parameter sets; see, e.g., \cite[Theorems~1.17]{RW09}. We include a short direct proof for
completeness.

\begin{lemma}\label{lem:s_continuity}
The function $s$ is continuous on $\R^n$.
\end{lemma}

\begin{proof}
Let
\[
    G(x,\lambda):=
    \left\|\sum_{i=1}^m\lambda_i\nabla f_i(x)\right\|,
    \quad \forall (x,\lambda)\in\R^n\times\Delta_m .
\]
Then $G$ is continuous and
\[
    s(x)=\min_{\lambda\in\Delta_m}G(x,\lambda).
\]
Let $x^j\to\bar x$. For any $\bar\lambda\in\Lambda(\bar x)$,
\[
    \limsup_{j\to\infty}s(x^j)
    \le
    \lim_{j\to\infty}G(x^j,\bar\lambda)
    =
    G(\bar x,\bar\lambda)
    =
    s(\bar x).
\]
For the reverse inequality, take a subsequence, denoted by $\{x^j\}$ for the sake of simplicity, such that $s(x^j)\to\liminf_j s(x^j)$. Choose $\lambda^j\in\Lambda(x^j)$. Since $\Delta_m$ is compact, passing to a further subsequence if necessary, we may assume that $\lambda^j\to\lambda^\ast\in\Delta_m$. Hence
\[
    \liminf_{j\to\infty}s(x^j)  = \lim_{j\to\infty}G(x^j,\lambda^j) = G(\bar x,\lambda^\ast) \ge s(\bar x).
\]
Therefore $\lim_{j\rightarrow\infty}s(x^j)= s(\bar x)$, and the proof is complete.
\end{proof}

    \section{Multi-objective Regularized Newton Method} \label{sec:algorithm}

    In this section, we introduce the multi-objective regularized Newton method. Recall that the stationarity measure $s(\cdot)$ and the multiplier set $\Lambda(\cdot)$ are defined in \eqref{eq:def_s}.

    For any $x \in \mathbb{R}^{n}$ and direction $d \in \mathbb{R}^{n}$, we approximate the objective functions by the local quadratic models
    \begin{equation} \label{eq:def_qi}
    q_{i}(d; x) := \langle\nabla f_{i}(x), d\rangle + \frac{1}{2}
        d^{\top} \nabla^{2} f_{i}(x) d, \quad \forall i\in [m].
    \end{equation}
    The quantity $q_i(d;x)$ is the second-order model of the increment $f_i(x+d)-f_i(x)$. Since each $f_i$ is convex, $\nabla^2 f_i(x)\succeq0$, and hence $q_i(\cdot;x)$ is convex.


    Based on these definitions, the detailed iterative procedure of the multi-objective regularized Newton method is formally presented in Algorithm \ref{Algo-MRNM}.

    \begin{algorithm}
        [H]
        \caption{Multi-objective Regularized Newton Method}
        \label{Algo-MRNM}
        \begin{algorithmic}
            [1] \Require Initial point $x^{0} \in \mathbb{R}^{n}$. \For{$k=0,1,2,\dots$}
            \State Compute the stationarity measure $s_{k} = s(x^{k})$. \If{$s_{k} = 0$}
            \State \textbf{stop} and return $x^{k}$. \EndIf 
            \State Set the regularization  parameter $\eta_{k} = 2\sqrt{H s_{k}}$.
            \State Compute the search direction $d^{k}$ by solving the subproblem: \State
            \begin{equation}
                \label{eq:def_subproblem}d^{k} = \argmin_{d\in\mathbb{R}^n}\left
                \{ \max_{i\in [m]}q_{i}(d; x^{k}) + \frac{\eta_{k}}{2}\|d\|^{2} \right
                \}
            \end{equation}
            \State Update the iterate: $x^{k+1}= x^{k} + d^{k}$. \EndFor
        \end{algorithmic}
    \end{algorithm}
    \begin{remark}
    {\rm (i)} The subproblem \eqref{eq:def_subproblem} in Algorithm~\ref{Algo-MRNM} is well defined. Indeed, whenever $s_k>0$, the regularization parameter $\eta_k=2\sqrt{Hs_k}$ is positive. Since each $q_i(\cdot;x^k)$ is convex, the function
    \[
        d\mapsto\max_{i\in[m]}q_i(d;x^k)+\frac{\eta_k}{2}\|d\|^2
    \]
    is strongly convex and coercive. Hence the subproblem has a unique minimizer, and the direction $d^k$ is uniquely determined.
    
    \noindent
    {\rm (ii)} When $m=1$, we have $s(x)=\|\nabla f_1(x)\|$, and the subproblem reduces to
    \[
    \min_{d\in\R^n} \left\{ \langle\nabla f_1(x^k),d\rangle +\frac12 d^\top\nabla^2 f_1(x^k)d +\sqrt{H \|\nabla f_1(x^k)\|}\|d\|^2 \right\}.
    \]
    Thus
    \[
    d^k = -\bigl(\nabla^2 f_1(x^k)+2\sqrt{H \|\nabla f_1(x^k)\|} I\bigr)^{-1}\nabla f_1(x^k),
    \]
    which is the square-root regularized Newton step \cite{mishchenko2023regularized}. For $m\ge2$, the method replaces this scalar quadratic model by the max-envelope of the component quadratic models, and therefore computes a common descent direction.
    \end{remark}

   To proceed with the convergence analysis, we denote the step length at iteration $k$ as
    \begin{equation}\label{eq:def_rk}
    r_k := \|d^k\| = \|x^{k+1} - x^k\|.
    \end{equation}
    We impose the following compactness assumption on the initial component-wise lower level set.

\begin{assumption}\label{as:compact}
    The initial level set $\mathcal{L}_0 := \operatorname{Lev}_F(x^0)$ is compact. That is, there exists a constant $D_0 > 0$ such that $\|x - y\| \leq D_0$ for all $x, y \in \mathcal{L}_0$.
\end{assumption}

In the single-objective case, Assumption~\ref{as:compact} reduces to the bounded level set condition used in \cite[Assumption~2.5]{mishchenko2023regularized}, up to the choice of the diameter constant. Thus it is a natural component-wise extension of that assumption to the multiobjective setting.

To measure the progress of the algorithm, we use the merit function proposed by Tanabe et al.~\cite{tanabe24}; see \eqref{eq:def_U}. Although the definition in \eqref{eq:def_U} is taken over the whole space $\R^n$, the descent property of Algorithm~\ref{Algo-MRNM} will be shown to keep the iterates inside the initial component-wise lower level set $\Lzero$. For the complexity analysis, we further need to ensure that $\Ugap$ is finite along the generated trajectory and to relate it to the stationarity measure $s(\cdot)$. The following proposition establishes these properties on the initial component-wise lower level set.

\begin{proposition}\label{thm:merit_properties}
Suppose that Assumption~\ref{as:compact} holds. For any $x\in\Lzero$, the following hold:
\begin{enumerate}
    \item[\rm (i)] $\mathcal{X}_w\cap\Lzero\neq\varnothing$, and
    \[
        \Ugap(x) = \sup_{z\in\mathcal{X}_w\cap\Lzero}  \min_{i\in[m]}\{f_i(x)-f_i(z)\}.
    \]
    In particular, $\Ugap(x)$ is finite.

    \item[\rm (ii)] $\Ugap(x)\le D_0\,s(x)$.
\end{enumerate}
\end{proposition}

\begin{proof}
    {\rm (i)} For any $z\in\Lzero$, define
    \[
    \mathcal S_z:=\{u\in\Lzero:\ f_i(u)\le f_i(z),\ i\in[m]\}.
    \]
    Then $\mathcal S_z$ is nonempty and compact. Let $\hat z$ minimize $\Phi(u):=\sum_{i=1}^m f_i(u)$ over $\mathcal S_z$. We claim that $\hat z\in\mathcal{X}_w$. Otherwise, there exists $y\in\R^n$ such that $F(y)<F(\hat z)$. Since $\hat z\in\mathcal S_z\subseteq\Lzero$, we have
    \[
    F(y)<F(\hat z)\le F(z)\le F(x^0),
    \]
    and hence $y\in\mathcal S_z$. This gives $\Phi(y)<\Phi(\hat z)$, contradicting the minimality of $\hat z$. Thus $\hat z\in\mathcal{X}_w\cap\Lzero$ and $F(\hat z)\le F(z)$. Taking $z=x^0$ shows that $\mathcal{X}_w\cap\Lzero\neq\varnothing$.

Now fix $x\in\Lzero$ and write
\[
    \psi_x(z):=\min_{i\in[m]}\{f_i(x)-f_i(z)\}.
\]
If $z\notin\Lzero$, then for some $j\in[m]$,
$f_j(z)>f_j(x^0)\ge f_j(x)$, and therefore
\[
    \psi_x(z)\le f_j(x)-f_j(z)<0.
\]
On the other hand, $\psi_x(x)=0$. Hence points outside $\Lzero$ cannot increase
the supremum in \eqref{eq:def_U}, and
\[
    \Ugap(x)=\sup_{z\in\Lzero}\psi_x(z).
\]
For every $z\in\Lzero$, the observation above gives
$\hat z\in\mathcal{X}_w\cap\Lzero$ with $F(\hat z)\le F(z)$. Hence, for every $i\in[m]$,
\[
    f_i(x)-f_i(z)\le f_i(x)-f_i(\hat z).
\]
Taking the minimum over $i\in[m]$ gives
\[
    \psi_x(z)\le \psi_x(\hat z) \le \sup_{\xi\in\mathcal{X}_w\cap\Lzero}\psi_x(\xi).
\]
Taking the supremum over $z\in\Lzero$ yields
\[
    \sup_{z\in\Lzero}\psi_x(z) \le \sup_{\xi\in\mathcal{X}_w\cap\Lzero}\psi_x(\xi).
\]
The reverse inequality follows from $\mathcal{X}_w\cap\Lzero\subseteq\Lzero$. Hence
\[
    \Ugap(x) = \sup_{z\in\mathcal{X}_w\cap\Lzero} \min_{i\in[m]}\{f_i(x)-f_i(z)\}.
\]
Since $\Lzero$ is compact and the component functions are continuous,
$\Ugap(x)$ is finite.

    \noindent
    {\rm (ii)} Fix $x \in \mathcal{L}_0$, and choose a multiplier $\omega \in \Lambda(x)$. By definition, $\left\|\sum_{i=1}^m \omega_i\nabla f_i(x)\right\| = s(x)$. Let $z \in \mathcal{X}_w \cap \mathcal{L}_0$. Utilizing the convexity of $f_i$, we obtain
    \[
        f_i(x) - f_i(z) \le \langle \nabla f_i(x), x - z\rangle, \quad \forall i \in [m].
    \]
    Since $\omega \in \Deltam$, we can bound the minimum difference as follows:
    \[
        \min_{i \in [m]} \{f_i(x) - f_i(z)\} \le \sum_{i=1}^m \omega_i \bigl(f_i(x) - f_i(z)\bigr) \le \left\langle \sum_{i=1}^m \omega_i\nabla f_i(x), x - z \right\rangle \le s(x)\|x - z\|.
    \]
    Because both $x, z \in \mathcal{L}_0$, Assumption \ref{as:compact} guarantees that $\|x - z\| \le D_0$. Taking the supremum over all $z \in \mathcal{X}_w \cap \mathcal{L}_0$ and using part (i) yields $\mathcal{U}(x) \le D_0\, s(x)$, completing the proof.
\end{proof}

The next lemma records the optimality condition of the subproblem in Algorithm~\ref{Algo-MRNM}; see, e.g., \cite[Section~3]{fliege09}.

\begin{lemma}\label{lem:kkt}
Let $d^k$ be the solution of \eqref{eq:def_subproblem}, and define the active index set
\[
    A_k:=\argmax_{i\in[m]} q_i(d^k;x^k).
\]
Then there exists $\alpha^k=(\alpha_1^k,\ldots,\alpha_m^k)\in\Deltam$, supported on $A_k$, such that
\begin{equation}\label{eq:kkt}
    0 = \sum_{i=1}^m \alpha_i^k \bigl(\nabla f_i(x^k)+\nabla^2 f_i(x^k)d^k\bigr) +\eta_k d^k .
\end{equation}
\end{lemma}

The following lemma provides basic estimates for one iteration of Algorithm~\ref{Algo-MRNM}. They connect the step length $r_k$, the regularization parameter $\eta_k$, and the stationarity measures $s_k$ and $s_{k+1}$.

    \begin{lemma}
        \label{lem:big_enough} Let $\{x^k\}_{k\in\mathbb N}$ be generated by Algorithm~\ref{Algo-MRNM}. Then, for every $k$ with $s_k>0$, the following estimates hold.
        \begin{align}
            \eta_{k} r_{k} & \le 2 s_{k}, \label{eq:eta_r_bound}                                       \\
            H r_{k}        & \le \frac{\eta_{k}}{2}, \label{eq:Hr_bound}                               \\
            s_{k+1}        & \le \frac{5}{4}\eta_{k} r_{k} \le \frac{5}{2}s_{k}. \label{eq:sk1_growth}
        \end{align}
    \end{lemma}

    \begin{proof}
        Since $d=0$ is feasible in \eqref{eq:def_subproblem}, the optimality yields
        \begin{equation}
            \label{eq:qi_ub}\max_{i} q_{i}(d^{k};x^{k}) + \frac{\eta_{k}}{2}\|d^{k}
            \|^{2} \le 0.
        \end{equation}
        By convexity, $\nabla^2 f_i(x^k)\succeq0$, and hence
        \[
        q_i(d^k;x^k)\ge \langle\nabla f_i(x^k),d^k\rangle, \quad i\in[m].
        \]
        Pick $\omega_{k} \in \Lambda(x^{k}).$ Therefore,
        \begin{align*}
            0 & \ge \max_{i} q_{i}(d^{k};x^{k}) + \frac{\eta_{k}}{2}\|d^{k}\|^{2} \ge \max_{i} \langle\nabla f_{i}(x^{k}),d^{k}\rangle + \frac{\eta_{k}}{2}r_{k}^{2}                                                                                                                   \\
              & \ge \sum_{i=1}^{m} \omega_{i}^{k}\langle\nabla f_{i}(x^{k}),d^{k}\rangle + \frac{\eta_{k}}{2}r_{k}^{2} = \left\langle\sum_{i=1}^{m} \omega_{i}^{k}\nabla f_{i}(x^{k}),d^{k}\right\rangle + \frac{\eta_{k}}{2}r_{k}^{2} \ge -s_{k} r_{k} + \frac{\eta_{k}}{2}r_{k}^{2},
        \end{align*}
        where the last inequality follows from the definition of $s_{k}.$ This
        proves \eqref{eq:eta_r_bound}.

        Since $\eta_{k}=2\sqrt{H s_{k}}$, we have $\eta_{k}^{2}=4Hs_{k}$. Multiplying
        \eqref{eq:eta_r_bound} by $H/\eta_{k}$ gives
        \[
            Hr_{k} \le \frac{2Hs_{k}}{\eta_{k}}= \frac{\eta_{k}}{2},
        \]
        which establishes \eqref{eq:Hr_bound}.

        It remains to bound $s_{k+1}$. Let $\alpha^k=(\alpha_1^k,\ldots,\alpha_m^k)\in\Deltam$ be the multiplier vector given by Lemma~\ref{lem:kkt}. Since $\alpha^k$ is feasible in the definition of $s(x^{k+1})$, we have
        \[
        s_{k+1} \le \left\| \sum_{i=1}^m \alpha_i^k\nabla f_i(x^{k+1})  \right\|,
        \]
        which by the triangle inequality implies
        \begin{align*}
            s_{k+1} \le \left\|\sum_{i=1}^{m} \alpha_{i}^{k}\bigl(\nabla f_{i}(x^{k+1}) - \nabla f_{i}(x^{k}) - \nabla^{2} f_{i}(x^{k})d^{k}\bigr)\right\| + \left\|\sum_{i=1}^{m} \alpha_{i}^{k}\bigl(\nabla f_{i}(x^{k}) + \nabla^{2} f_{i}(x^{k})d^{k}\bigr)\right\|.
        \end{align*}
        By Lemma~\ref{lem:upperbound} and \eqref{eq:kkt}, we obtain
        \[
            s_{k+1}\le \frac{H}{2}r_{k}^{2} + \eta_{k} r_{k}.
        \]
        Using \eqref{eq:Hr_bound}, we get $\frac{H}{2}r_{k}^{2} \le \frac{\eta_{k}}{4}
        r_{k}$, hence
        \[
            s_{k+1}\le \frac{5}{4}\eta_{k} r_{k}.
        \]
        Combining this with \eqref{eq:eta_r_bound} proves the second inequality in
        \eqref{eq:sk1_growth}.
    \end{proof}

    The estimates above imply a component-wise descent property, which keeps the iterates in the initial level set and yields a one-step decrease of the merit function.

    \begin{lemma}\label{lem:component-wise}
Suppose Assumption~\ref{as:compact} holds, and let $\{x^k\}_{k\in\mathbb N}$ be
generated by Algorithm~\ref{Algo-MRNM}. Then, for every nonterminal iteration
$k$ with $s_k>0$, the following statements hold:
\begin{itemize}
    \item[\rm (i)] For every $i\in[m]$,
    \[
        f_i(x^{k+1})
        \le
        f_i(x^k)-\frac{5}{12}\eta_k r_k^2.
    \]
    Consequently, $F(x^{k+1})\le F(x^k)$, and hence $x^k\in\Lzero$ for all
    generated iterates.

    \item[\rm (ii)] Let $u_k:=\Ugap(x^k)$. Then
    \[
        u_{k+1}
        \le
        u_k-\frac{5}{12}\eta_k r_k^2.
    \]
\end{itemize}
\end{lemma}

\begin{proof}
{\rm (i)} Fix $i\in[m]$. Since $x^{k+1}=x^k+d^k$, Lemma~\ref{lem:upperbound}
gives
\[
    f_i(x^{k+1}) \le f_i(x^k)+q_i(d^k;x^k)+\frac{H}{6}r_k^3.
\]
Moreover, since $d=0$ is feasible in \eqref{eq:def_subproblem}, the optimality
of $d^k$ yields
\[
    \max_{j\in[m]}q_j(d^k;x^k)+\frac{\eta_k}{2}r_k^2\le0.
\]
Thus we get
\[
    q_i(d^k;x^k) \le \max_{j\in[m]}q_j(d^k;x^k) \le -\frac{\eta_k}{2}r_k^2.
\]
Combining the last two estimates, we obtain
\[
    f_i(x^{k+1}) \le f_i(x^k)-\frac{\eta_k}{2}r_k^2+\frac{H}{6}r_k^3.
\]
By \eqref{eq:Hr_bound}, $Hr_k\le\eta_k/2$, and hence
\[
    \frac{H}{6}r_k^3 \le \frac{\eta_k}{12}r_k^2.
\]
Therefore, it holds that
\[
    f_i(x^{k+1}) \le f_i(x^k)-\frac{5}{12}\eta_k r_k^2.
\]
This proves the component-wise descent $F(x^{k+1})\leq F(x^k)$. Since $x^0\in\Lzero$, the inclusion $x^k\in\Lzero$ for all generated iterates follows by induction.

\noindent
{\rm (ii)} Fix any $z\in\R^n$. By part {\rm (i)}, for every $i\in[m]$,
\[
    f_i(x^{k+1})-f_i(z)
    \le
    f_i(x^k)-f_i(z)-\frac{5}{12}\eta_k r_k^2.
\]
Taking the minimum over $i\in[m]$ gives
\[
    \min_{i\in[m]}\{f_i(x^{k+1})-f_i(z)\}
    \le
    \min_{i\in[m]}\{f_i(x^k)-f_i(z)\}
    -\frac{5}{12}\eta_k r_k^2.
\]
Since this inequality holds for every $z\in\R^n$, the definition of $\Ugap$ implies
\[
    u_{k+1}
    \le
    u_k-\frac{5}{12}\eta_k r_k^2.
\]
\end{proof}
    
 We now distinguish iterations according to the behavior of the stationarity measure. The following set collects the iterations at which $s_{k+1}$ does not decrease too much relative to $s_k$:
\begin{equation}\label{eq:def_good}
    \mathcal I_\infty
    :=
    \left\{k\in\mathbb N:\ s_{k+1}\ge \frac14 s_k\right\}.
\end{equation}
For each $k\ge1$, define the truncated set
\begin{equation}\label{eq:def_Ik}
    \mathcal I_k
    :=
    \mathcal I_\infty\cap\{0,1,\ldots,k-1\}.
\end{equation}
The next lemma shows that every iteration in $\mathcal I_\infty$ produces a quantitative decrease of the merit function.

\begin{lemma}\label{lem:good_decrease}
Suppose Assumption~\ref{as:compact} holds. If $k\in\mathcal I_\infty$, then
\begin{align}
    \frac{5}{12}\eta_k r_k^2
    &\ge
    \frac{s_k^{3/2}}{120\sqrt H}, \label{eq:I_step_stationarity_decrease}\\
    u_{k+1}
    \le
    u_k-\tau & u_k^{3/2},
    \quad
    \tau:=\frac{1}{120D_0^{3/2}\sqrt H}. \label{eq:good_step_merit}
\end{align}
\end{lemma}

\begin{proof}
Let $k\in\mathcal I_\infty$. By the definition of $\mathcal I_\infty$ and
\eqref{eq:sk1_growth},
\[
    \frac14 s_k
    \le
    s_{k+1}
    \le
    \frac54\eta_k r_k.
\]
Hence
\[
    \eta_k r_k\ge \frac15 s_k.
\]
Since $\eta_k=2\sqrt{Hs_k}$, we also obtain
\[
    r_k
    \ge
    \frac{s_k}{5\eta_k}
    =
    \frac{\sqrt{s_k}}{10\sqrt H}.
\]
Combining the two bounds gives
\[
    \frac{5}{12}\eta_k r_k^2
    =
    \frac{5}{12}(\eta_k r_k)r_k
    \ge
    \frac{5}{12}\cdot\frac{s_k}{5}
    \cdot
    \frac{\sqrt{s_k}}{10\sqrt H}
    =
    \frac{s_k^{3/2}}{120\sqrt H}.
\]
This proves \eqref{eq:I_step_stationarity_decrease}.

By Proposition~\ref{thm:merit_properties}{\rm (ii)}, $u_k\le D_0s_k$, and therefore
\[
    s_k^{3/2}
    \ge
    \frac{u_k^{3/2}}{D_0^{3/2}}.
\]
Consequently,
\[
    \frac{5}{12}\eta_k r_k^2
    \ge
    \frac{u_k^{3/2}}{120D_0^{3/2}\sqrt H}
    =
    \tau u_k^{3/2}.
\]
Combining this estimate with Lemma~\ref{lem:component-wise}{\rm (ii)} gives
\[
    u_{k+1}
    \le
    u_k-\tau u_k^{3/2},
\]
which proves \eqref{eq:good_step_merit}.
\end{proof}

We now turn to the asymptotic merit complexity of Algorithm~\ref{Algo-MRNM}. The following localization lemma is the key step in the proof of the global little-$o$ rate.

\begin{lemma}\label{lem:U_little_o_s}
Suppose that Assumption~\ref{as:compact} holds. If the algorithm does not terminate finitely, then
\[
    \mathcal U(x^k)=o(s_k).
\]
\end{lemma}
\begin{proof}
Let
\[
    u_k:=\mathcal U(x^k).
\]
Since the algorithm does not terminate finitely, we have $s_k>0$ for all $k$.
By Lemma~\ref{lem:component-wise}, the sequence $\{u_k\}$ is nonincreasing and
nonnegative. Hence it admits a limit.

We first show that
\begin{equation}\label{eq-lim-uk-zero}
    \lim_{k\rightarrow\infty}u_k =  0.
\end{equation}
If $\mathcal I_\infty$ is finite, then for all sufficiently large $k$ one has
$k\notin\mathcal I_\infty$, and hence
\[
    s_{k+1}<\frac14 s_k.
\]
Thus $s_k\to0$, and Proposition~\ref{thm:merit_properties}(ii) gives $u_k\le D_0s_k\to0$.

If $\mathcal I_\infty$ is infinite, then Lemma~\ref{lem:good_decrease} gives
\[
    u_{k+1}\le u_k-\tau u_k^{3/2},
    \quad k\in\mathcal I_\infty .
\]
If $\lim_k u_k=\bar u>0$, then every sufficiently large $k\in \mathcal{I}_{\infty}$ would decrease $u_k$ by at least $\frac{1}{2}\tau \bar u^{3/2}>0$, which is impossible since
$\{u_k\}_{k\in\mathbb{N}}$ is nonincreasing and bounded below. Therefore $\bar u=0$. Hence, in
all cases, equation \eqref{eq-lim-uk-zero} holds.

Since $x^k\in\Lzero$ for all $k$ and $\Lzero$ is compact, every subsequence of $\{x^k\}$ has a cluster point. Let $x^{k_j}\to\bar x$. If $\bar x\notin\mathcal{X}_w$, then there exists $y\in\R^n$ such that $f_i(y)<f_i(\bar x)$ for all $i\in[m]$. Hence, for some $\delta>0$ and all
sufficiently large $j$,
\[
    f_i(x^{k_j})-f_i(y)\ge \delta,
    \qquad i\in[m],
\]
which gives $u_{k_j}\ge\delta$, contradicting \eqref{eq-lim-uk-zero}. Thus every cluster point of $\{x^k\}$ belongs to $\mathcal{X}_w$. We now show that
\begin{equation}\label{eq:stationarity_to_zero}
   \lim_{k\rightarrow\infty} s_k= 0.
\end{equation}
Indeed, if not, there exist $\varepsilon>0$ and a subsequence $\{k_j\}$ such that $s_{k_j}\ge\varepsilon$. Since $\{x^k\}\subset\Lzero$ and $\Lzero$ is compact, we may pass to a further subsequence such that $x^{k_j}\to\bar x$. The preceding argument gives $\bar x\in\mathcal{X}_w$, and hence $s(\bar x)=0$ by Lemma~\ref{lem:weak_pareto_characterization}. By Lemma~\ref{lem:s_continuity}, $s$ is continuous, and hence
$s_{k_j}\to s(\bar x)=0$, contradicting $s_{k_j}\ge\varepsilon$. Thus, equation \eqref{eq:stationarity_to_zero} is true. 

By Lemma~\ref{lem:component-wise}, each sequence $\{f_i(x^k)\}$ is
nonincreasing. Since $\{x^k\}\subset\Lzero$ and $\Lzero$ is compact, the limits
\[
    \ell_i:=\lim_{k\to\infty}f_i(x^k),
    \quad i\in[m],
\]
exist. Moreover, every cluster point $\bar x$ of $\{x^k\}$ satisfies
$f_i(\bar x)=\ell_i$ for all $i\in[m]$.

We now prove $u_k=o(s_k)$ by contradiction. Suppose that there are $\varepsilon>0$ and a subsequence, still denoted by $\{k_j\}$, such that
\[
    u_{k_j}\ge \varepsilon s_{k_j}.
\]
Passing to a further subsequence if necessary, assume that $x^{k_j}\to\bar x$. Choose $\lambda^{k_j}\in\Lambda(x^{k_j})$. Since $\Deltam$ is compact, we may also assume that $\lambda^{k_j}\to\bar\lambda\in\Deltam$. Since $\lambda^{k_j}\in\Lambda(x^{k_j})$, the definition of $\Lambda$ gives
\[
    \left\|
        \sum_{i=1}^m\lambda_i^{k_j}\nabla f_i(x^{k_j})
    \right\|
    =
    s_{k_j}.
\]
Letting $j\to\infty$ and using \eqref{eq:stationarity_to_zero},
$x^{k_j}\to\bar x$, and $\lambda^{k_j}\to\bar\lambda$, we obtain
\[
    \sum_{i=1}^m\bar\lambda_i\nabla f_i(\bar x)=0.
\]
Let $S:=\{i\in[m]:\bar\lambda_i>0\}$. Then $S\ne\varnothing$. The convex function
\[
    \phi_{\bar\lambda}(z):=\sum_{i=1}^m\bar\lambda_i f_i(z)
\]
has zero gradient at $\bar x$, and hence $\bar x$ is a global minimizer of
$\phi_{\bar\lambda}$. Therefore, for every $z\in\R^n$, there exists
$i(z)\in S$ such that
\[
    f_{i(z)}(z)\ge f_{i(z)}(\bar x)=\ell_{i(z)}.
\]
Consequently,
\[
    \min_{r\in[m]}\{f_r(x^{k_j})-f_r(z)\}  \le  f_{i(z)}(x^{k_j})-f_{i(z)}(z) \le \max_{i\in S}\{f_i(x^{k_j})-\ell_i\}.
\]
Taking the supremum over $z\in\R^n$ gives
\begin{equation}\label{eq:merit_by_active_values}
    u_{k_j} \le \max_{i\in S}\{f_i(x^{k_j})-\ell_i\}.
\end{equation}

 Since $\bar\lambda_i>0$ for $i\in S$, there exists $\gamma>0$ such that $\lambda_i^{k_j}\ge\gamma$ for all $i\in S$ and all sufficiently large $j$. Together with $f_i(x^{k_j})-\ell_i>0$, equation \eqref{eq:merit_by_active_values} implies
\[
    u_{k_j} \le \frac1\gamma \sum_{i=1}^m \lambda_i^{k_j}\bigl(f_i(x^{k_j})-\ell_i\bigr).
\]
Using $f_i(\bar x)=\ell_i$ and the convexity of $f_i$, we obtain
\[
    f_i(x^{k_j})-f_i(\bar x) \le \langle \nabla f_i(x^{k_j}),x^{k_j}-\bar x\rangle .
\]
Therefore,
\[
    u_{k_j} \le \frac1\gamma \left\langle  \sum_{i=1}^m\lambda_i^{k_j}\nabla f_i(x^{k_j}),  x^{k_j}-\bar x \right\rangle \le \frac1\gamma s_{k_j}\|x^{k_j}-\bar x\|.
\]
Dividing by $s_{k_j}>0$ and letting $j\to\infty$ contradicts $u_{k_j}\ge\varepsilon s_{k_j}$. The proof is complete.
\end{proof}

We now derive the main merit complexity estimate. The result in Lemma~\ref{lem:U_little_o_s} is the additional ingredient that sharpens the usual $\OO(1/k^2)$ complexity to the asymptotic $o(1/k^2)$ bound below.

\begin{theorem}[Global $o(1/k^2)$ rate for the merit function]\label{th:little_o_rate}
Suppose Assumption~\ref{as:compact} holds, and let $\{x^k\}$ be generated by Algorithm~\ref{Algo-MRNM}. If the algorithm does not terminate in finitely many iterations, then
\begin{equation}\label{eq:little_o_rate}
    \Ugap(x^k)=o(1/k^2).
\end{equation}
If the algorithm terminates, the terminal point is weakly Pareto optimal.
\end{theorem}

\begin{proof}
If the algorithm terminates at some iteration $j$, then $s_j=0$, and Lemma~\ref{lem:weak_pareto_characterization} implies $x^j\in\mathcal{X}_w$, for which the desired result automatically holds. We therefore assume that the algorithm does not terminate in finite iterations. In particular, $s_k>0$ for every $k$. Let $u_k:=\Ugap(x^k)$. By the definition of $\Ugap$ and Lemma~\ref{lem:weak_pareto_characterization}, $u_k>0$ for every $k$.

Set
\[
    \rho_k:=\frac{u_k}{s_k}.
\]
By Lemma \ref{lem:U_little_o_s}, $\rho_k\to0$. For each $k\in\mathcal I_\infty$, Lemma~\ref{lem:good_decrease} gives
\[
    \frac{5}{12}\eta_k r_k^2
    \ge
    \frac{s_k^{3/2}}{120\sqrt H}.
\]
Combining this estimate with Lemma~\ref{lem:component-wise}{\rm (ii)}, we get
\begin{equation}\label{eq:good_step_stationarity_little_o}
    u_{k+1}
    \le
    u_k-\kappa s_k^{3/2},
    \qquad
    \kappa:=\frac{1}{120\sqrt H}.
\end{equation}
Hence, for $k\in\mathcal I_\infty$,
\[
    u_{k+1}
    \le
    u_k-\kappa\rho_k^{-3/2}u_k^{3/2}.
\]

If $\mathcal I_\infty$ is finite, by the definition of $\mathcal I_\infty$, the sequence $\{s_k\}$ then decays geometrically from some index onward. Proposition~\ref{thm:merit_properties} gives $u_k\le D_0s_k$, and hence \eqref{eq:little_o_rate} follows. It remains to consider the case in which $\mathcal I_\infty$ is infinite. Write
\[
    \mathcal I_\infty :=\{i_0<i_1<i_2<\cdots\}, \quad v_t:=u_{i_t}, \quad b_t:=\kappa\rho_{i_t}^{-3/2}.
\]
Since $\rho_k\to0$, we have $b_t\to\infty$. Moreover, $b_t>0$ and $v_t\ge0$. By the monotonicity of $\{u_k\}$ and \eqref{eq:good_step_stationarity_little_o},
\[
    v_{t+1}  =  u_{i_{t+1}} \le u_{i_t+1} \le v_t-b_t v_t^{3/2}.
\]
Thus the assumptions of Lemma~\ref{lem:variable_scalar_recursion} are satisfied, and hence
\begin{equation}\label{eq:good_subsequence_little_o}
    v_t=o(1/t^2).
\end{equation}
We now pass from the subsequence indexed by $\mathcal I_\infty$ to the full
sequence. For $k\ge1$, define
\[
    N_k:=|\mathcal I_k|.
\]
If $N_k\le k/2$, then using \eqref{eq:sk1_growth} on the indices in
$\mathcal I_\infty$ and the definition of $\mathcal I_\infty$ on the remaining
indices, we obtain
\[
    s_k
    \le
    \left(\frac52\right)^{N_k}
    \left(\frac14\right)^{k-N_k}s_0
    \le
    \left(\sqrt{\frac58}\right)^k s_0 .
\]
Hence Proposition~\ref{thm:merit_properties}{\rm (ii)} gives
\[
    u_k\le D_0s_k=o(1/k^2).
\]

It remains to consider the indices for which $N_k>k/2$. Let $t=N_k-1$.
Then $i_t$ is the last index in $\mathcal I_\infty$ before $k$. By the
monotonicity of $\{u_k\}$,
\[
    u_k\le u_{i_t}=v_t.
\]
Moreover, for all sufficiently large $k$,
\[
    t=N_k-1>k/2-1\ge k/4.
\]
Therefore \eqref{eq:good_subsequence_little_o} implies
\[
    u_k\le v_t=o(1/t^2)=o(1/k^2).
\]
The two cases prove \eqref{eq:little_o_rate}.
\end{proof}

\begin{remark}
Theorem~\ref{th:little_o_rate} immediately implies the global $\OO(1/k^2)$ merit complexity bound. When $m=1$, the merit function reduces to the objective gap,
\[
    \Ugap(x)=f_1(x)-\min_{z\in\R^n}f_1(z),
\]
while Algorithm~\ref{Algo-MRNM} reduces to the square-root regularized Newton method. Thus Theorem~\ref{th:little_o_rate} yields
\[
    f_1(x^k)-\min_{z\in\mathbb{R}^n} f_1(z)=o(1/k^2),
\]
which refines the scalar $\OO(1/k^2)$ guarantee in \cite{mishchenko2023regularized} to a little-$o$ statement. 
\end{remark}

\section{Essential Sharpness of the Polynomial Exponent}
\label{sec:barrier}

In this section, we show that the exponent $2$ in Theorem~\ref{th:little_o_rate} is essentially unimprovable in a uniform worst-case sense. The construction is one-dimensional and bi-objective, but along the generated trajectory the multi-objective subproblem reduces to a scalar regularized Newton step. 

\subsection{A one-dimensional barrier family}
Fix $p>2$. Define the even function $\psi_p:\R\to\R$ by
\[
    \psi_p(x):=
    \begin{cases}
        \dfrac{|x|^{p+1}}{p+1}, & |x|\le1,\\[0.4em]
        \dfrac{1}{p+1}+(|x|-1)+\dfrac{p}{2}(|x|-1)^2, & |x|\ge1.
    \end{cases}
\]
Then $\psi_p\in C^2(\R)$ is convex and coercive, and $\psi_p''$ is Lipschitz continuous with constant $H_p:=p(p-1)$. Consider
\begin{equation}\label{eq:family}
    \min_{x\in\R}\bigl(f_1(x),f_2(x)\bigr), \quad f_1(x):=\psi_p(x),\quad f_2(x):=2x+\psi_p(x).
\end{equation}
Fix $x^0\in(0,1)$ and define $\Lzero$ accordingly. For $x\in(0,1)$,
\[
    f_1'(x)=x^p,\quad f_2'(x)=2+x^p,\quad f_1''(x)=f_2''(x)=p x^{p-1}.
\]
Therefore
\begin{equation}\label{eq:skexpression}
    s(x)=\min_{\lambda\in[0,1]}
    \left|\lambda x^p+(1-\lambda)(2+x^p)\right|=x^p.
\end{equation}
Moreover, $0\in\mathcal{X}_w$, and no positive point is weak Pareto optimal because $0$ strictly improves both objectives at every $x>0$. Hence every $z\in\mathcal{X}_w\cap\Lzero$ satisfies $z\le0$. For $x\in(0,1)$ and $z\in\mathcal{X}_w\cap\Lzero$,
\[
    f_2(x)-f_2(z)  =   f_1(x)-f_1(z)+2(x-z) >  f_1(x)-f_1(z).
\]
Thus
\[
    \min\{f_1(x)-f_1(z),f_2(x)-f_2(z)\}=f_1(x)-f_1(z).
\]
Since $f_1$ is minimized at $0$ and $0\in\mathcal{X}_w\cap\Lzero$, Proposition~\ref{thm:merit_properties}(i) gives
\[
    \Ugap(x)=f_1(x)-f_1(0)=\frac{x^{p+1}}{p+1}.
\]

We next examine the subproblem. For $x\in(0,1)$, define
\[
    \phi_x(d) := \max\{q_1(d;x),q_2(d;x)\} +\frac{\eta}{2}d^2,
    \quad \eta>0.
\]
Since
\[
    q_2(d;x)=q_1(d;x)+2d,
\]
we have $q_2(d;x)\ge q_1(d;x)$ for $d\ge0$ and
$q_2(d;x)<q_1(d;x)$ for $d<0$. If $d\ge0$, then
\[
    \phi_x(d)
    =
    q_2(d;x)+\frac{\eta}{2}d^2
    =
    (2+x^p)d+\frac12 p x^{p-1}d^2+\frac{\eta}{2}d^2
    \ge0.
\]
On the other hand, for all sufficiently small $d<0$,
\[
    \phi_x(d)
    =
    q_1(d;x)+\frac{\eta}{2}d^2
    =
    x^p d+\frac12 p x^{p-1}d^2+\frac{\eta}{2}d^2
    <0.
\]
Therefore, the minimizer must lie in $(-\infty,0)$, and the active model is $q_1$, so the subproblem reduces to the scalar quadratic problem
\begin{equation}\label{eq:dkexpression}
    \min_{d<0}
    \left\{
        x^p d+\frac12 p x^{p-1}d^2+\frac{\eta}{2}d^2
    \right\}.
\end{equation}

\subsection{Asymptotic rates of the recursion}
To clarify the role of the square-root choice and to show the sharpness barrier in a slightly broader form, we analyze the more general regularization rule
\[
    \eta_k=C\,s(x^k)^\alpha,\quad {\rm with} \ \ C>0,\quad \alpha\in(0,1].
\]
The square-root rule used in Algorithm~\ref{Algo-MRNM} corresponds to
$\alpha=1/2$ and $C=2\sqrt{H_p}$. For a current point $x\in(0,1)$, recall that $s(x)$ is derived in \eqref{eq:skexpression}, and this combines with \eqref{eq:dkexpression} gives
\begin{equation}\label{eq:general-step}
    d(x)
    =
    -\frac{x^p}{p x^{p-1}+C x^{p\alpha}},
\end{equation}
and the trajectory is $x^{k+1}=x^k+d(x^k)$.

\begin{theorem}\label{th:power-law-barrier}
Fix $\alpha\in(0,1]$ and $C>0$. For the family \eqref{eq:family}, consider the iteration generated by \eqref{eq:general-step} with initial point $x^0 \in (0,1)$. Then:
\begin{itemize}
    \item[\rm (i)] If $\alpha<1-\frac1p$, then $x^k\downarrow0$ and there exist
    constants $c_2>c_1>0$ and $K\ge1$ such that, for all $k\ge K$,
    \[
        c_1 k^{-q_{p,\alpha}}
        \le
        \Ugap(x^k)
        \le
        c_2 k^{-q_{p,\alpha}},
        \qquad
        q_{p,\alpha}:=\frac{p+1}{p(1-\alpha)-1}.
    \]
    \item[\rm (ii)] If $\alpha=1-\frac1p$, then 
    \[
        x^{k+1}=\left(1-\frac{1}{p+C}\right)x^k.
    \]
    \item[\rm (iii)] If $\alpha>1-\frac1p$, then 
    \[
        \lim_{k\to\infty}\frac{x^{k+1}}{x^k}=1-\frac1p.
    \]
\end{itemize}
\end{theorem}

 \begin{proof}
From \eqref{eq:general-step},
\[
    \frac{-d(x)}{x}
    =
    \frac{1}{p+C x^{p\alpha-p+1}},
\]
which implies that $0<-d(x)<x/p$ for all $x\in(0,1)$. Since $x^0 \in (0,1)$ and $p>2$, the iterates remain in $(0,1)$ and strictly decrease. Their limit must be $0$, since any positive limit would imply a nonzero limiting decrease.

{\rm (i)} If $\alpha<1-\frac1p$, then $p\alpha<p-1$. Thus
\[
    \lim_{x\searrow0}\frac{-d(x)}{x^{p(1-\alpha)}}=\frac1C.
\]
Set $\beta_{p,\alpha}:=p(1-\alpha)>1$. Lemma~\ref{lem:scalar-asymp} gives
\[
    \lim_{k\to\infty}
    k^{1/(\beta_{p,\alpha}-1)}x^k
    =
    \bigl(C^{-1}(\beta_{p,\alpha}-1)\bigr)^{-1/(\beta_{p,\alpha}-1)}.
\]
Since $\Ugap(x^k)=(x^k)^{p+1}/(p+1)$, the claimed two-sided bound follows.

{\rm (ii)} If $\alpha=1-\frac1p$, then $p\alpha=p-1$, and
\[
    d(x)=-\frac{x}{p+C}.
\]
This gives the stated linear recursion.

{\rm (iii)} If $\alpha>1-\frac1p$, then $p\alpha>p-1$, and therefore
\[
    \lim_{x\searrow0}\frac{-d(x)}{x}=\frac1p.
\]
Since $x^k\to0$, this implies
\[
    \lim_{k\to\infty}\frac{x^{k+1}}{x^k}
    =
    1-\frac1p.
\]
\end{proof}

We now apply the preceding result to the square-root regularization rule used in Algorithm~\ref{Algo-MRNM}. This gives the rate on the constructed trajectories and shows that its exponent can be arbitrarily close to \(2\).

\begin{corollary}\label{cor:square-root}
Suppose $p>2$ and use the square-root rule
\[
    \eta_k=2\sqrt{H_p s(x^k)}.
\]
Then there exist constants $c_2>c_1>0$ and an integer $K\ge1$ such that, for all
$k\ge K$,
\[
    c_1 k^{-q_p} \le \Ugap(x^k)  \le  c_2 k^{-q_p}, \qquad  q_p:=\frac{2(p+1)}{p-2}  = 2+\frac{6}{p-2}.
\]
Consequently, for every $\delta>0$, no uniform estimate of the form
\[
    \Ugap(x^k)=\OO(k^{-(2+\delta)})
\]
can hold over this family.
\end{corollary}

Theorem~\ref{th:little_o_rate} and Corollary~\ref{cor:square-root} together clarify the role of the exponent \(2\) for the square-root regularized Newton method. The merit function satisfies the global asymptotic bound \(o(1/k^2)\) along every generated sequence, while the barrier family shows that this exponent cannot be improved in a uniform polynomial sense.

\section{Conclusion}

In this paper, we studied a regularized Newton method for unconstrained convex multi-objective optimization. Under the compactness of the initial component-wise lower level set and Lipschitz continuity of the Hessians, we established the global asymptotic merit estimate $\Ugap(x^k)=o(1/k^2)$. In the scalar specialization, the little-$o$ estimate improves the $\OO(1/k^2)$ rate for the regularized Newton method established by Mishchenko~\cite{mishchenko2023regularized}.

We also constructed a one-dimensional bi-objective family showing that the exponent $2$ is essentially sharp in a uniform sense, that is, for any fixed $\delta>0$, no uniform estimate of order $\OO(k^{-(2+\delta)})$ holds over this family. Thus, the $o(1/k^2)$ refinement does not imply a uniform improvement to any higher fixed polynomial rate.

\bibliographystyle{siam}
\bibliography{references}

\appendix
\section{Two auxiliary lemmas}

The first lemma is a simple variant of the standard scalar recursion estimates used in regularized Newton complexity analysis.

\begin{lemma}\label{lem:variable_scalar_recursion}
Let $\{v_t\}_{t\ge0}$ be a nonnegative sequence and let
$\{b_t\}_{t\ge0}$ be a positive sequence with $b_t\to\infty$. Suppose that
\[
    v_{t+1}\le v_t-b_t v_t^{3/2}, \quad \text{for all }t\ge0,
\]
Then
\[
    v_t=o(1/t^2).
\]
\end{lemma}

\begin{proof}
If $v_t=0$ for some $t$, the conclusion is immediate. Otherwise, $v_t>0$ for
all $t$. Since
\[
    0<v_{t+1}\le v_t-b_t v_t^{3/2} =v_t(1-b_t\sqrt{v_t}),
\]
we have $b_t\sqrt{v_t}\in(0,1)$ for every $t$.

Using $(1-a)^{-1/2}\ge 1+a/2$ for $a\in(0,1)$, we obtain
\[
    \frac1{\sqrt{v_{t+1}}} \ge \frac1{\sqrt{v_t-b_t v_t^{3/2}}} = \frac1{\sqrt{v_t}}\frac1{\sqrt{1-b_t\sqrt{v_t}}} \ge \frac1{\sqrt{v_t}}+\frac{b_t}{2}.
\]
Summing from $0$ to $t-1$ yields
\[
    v_t  \le  \frac{4}{\left(\sum_{\ell=0}^{t-1}b_\ell\right)^2}.
\]
Since $b_t\to\infty$,
\[
    \lim_{t\to\infty}\frac1t\sum_{\ell=0}^{t-1}b_\ell = \infty.
\]
Therefore,
\[
    t^2 v_t \le \frac{4}{\left(t^{-1}\sum_{\ell=0}^{t-1}b_\ell\right)^2} \to0,
\]
which proves $v_t=o(1/t^2)$.
\end{proof}

The next lemma turns an asymptotic one-step recursion into the corresponding polynomial decay rate. Its proof is elementary and is included for completeness.
    \begin{lemma}
        \label{lem:scalar-asymp} Let $\{t_{k}\}$ be a positive sequence with $\lim_{k\to\infty}t_{k} = 0$ and suppose there exist constants $c>0$ and $\beta >1$ such that
        \begin{equation}\label{A2condition}
            \lim_{k\to\infty}\frac{t_{k} - t_{k+1}}{t_{k}^{\beta}}= c.
        \end{equation}
        Then the sequence satisfies
        \[
            \lim_{k\to\infty}k^{\frac{1}{\beta-1}}t_{k} = \bigl(c(\beta-1)\bigr)^{-\frac{1}{\beta-1}}.
        \]
    \end{lemma}

    \begin{proof}
Set
\[
    \gamma:=\beta-1>0, \quad y_k:=t_k^{-\gamma}.
\]
Since $t_k\to0$, we have $y_k\to\infty$. By the hypothesis \eqref{A2condition}, we write
\[
    t_{k+1} = t_k-c t_k^\beta+r_k,
\]
where $r_k$ satisfies $\lim_{k\rightarrow \infty} \frac{r_k}{t_k^\beta} = 0.$ Dividing both sides by $t_k$, we get
\[
    \frac{t_{k+1}}{t_k}  = 1-c t_k^\gamma+\frac{r_k}{t_k} = 1-u_k, \ \ {\rm with} \ \ u_k:=c t_k^\gamma-\frac{r_k}{t_k}.
\]
From the definition of $u_k$,
\[
    \frac{u_k}{t_k^\gamma}  =  c-\frac{r_k}{t_k^\beta}.
\]
Since $\frac{r_k}{t_k^\beta} \to 0$, we have $ \frac{u_k}{t_k^\gamma} \to c$, for which $u_k\to0$, and $u_k>0$ for all sufficiently large $k$.

Note that
\[
\begin{aligned}
    y_{k+1}-y_k &= t_{k+1}^{-\gamma}-t_k^{-\gamma}   = t_k^{-\gamma} \left[ \left(\frac{t_{k+1}}{t_k}\right)^{-\gamma}-1 \right] = t_k^{-\gamma} \left[  (1-u_k)^{-\gamma}-1  \right].
\end{aligned}
\]
Since $u_k\to0$, the Taylor expansion of $(1-u)^{-\gamma}$ at $u=0$
gives
\[
    (1-u_k)^{-\gamma}-1 = \gamma u_k+\rho_k u_k, \quad {\rm with} \ \ \lim_{k\rightarrow \infty}\rho_k = 0.
\]
Therefore,
\[
\begin{aligned}
    y_{k+1}-y_k &= t_k^{-\gamma} \left(\gamma u_k+\rho_k u_k\right) = \left(\gamma+\rho_k\right)  \frac{u_k}{t_k^\gamma}.
\end{aligned}
\]
Using $u_k/t_k^\gamma\to c$ and $\rho_k\to0$, we obtain
\[
   \lim_{k\rightarrow\infty} y_{k+1}-y_k =  \gamma c.
\]
Thus,
\[
    \frac{y_k}{k}  = \frac{y_0}{k} + \frac1k\sum_{\ell=0}^{k-1}(y_{\ell+1}-y_\ell) \to \gamma c.
\]
Since $y_k=t_k^{-\gamma}$, we have
\[
    \lim_{k\rightarrow\infty} k t_k^\gamma = \frac1{\gamma c},
\]
and hence
\[
   \lim_{k\rightarrow\infty}  k^{1/\gamma}t_k = (\gamma c)^{-1/\gamma}.
\]
Finally, substituting $\gamma=\beta-1$ gives
\[
  \lim_{k\rightarrow\infty}  k^{\frac1{\beta-1}}t_k =  \bigl(c(\beta-1)\bigr)^{-\frac1{\beta-1}},
\]
and this proves the desired result.
\end{proof}
\end{document}